\documentclass{llncs}
\usepackage{amsfonts}
\usepackage{amssymb}
\usepackage{graphicx}
\usepackage{psfrag}
\usepackage{color}
\usepackage{enumerate}
\usepackage[all]{xy}
\newcommand{\R}{{\mathbb{R}}}
\begin{document}

%\begin{frontmatter}

\title{Topological graph persistence}

\titlerunning{Topological graph persistence}  % abbreviated title (for running head)
\author{Mattia G. Bergomi$^1$, Massimo Ferri$^2$, Lorenzo Zuffi$^2$}
\authorrunning{M.G. Bergomi, M. Ferri, L. Zuffi} % abbreviated author list (for running head)
%\tocauthor{author1, author2}
%
\institute{$^1$ Champalimaud Center for the Unknown, Lisbon, Portugal\\
$^2$ ARCES and Dept. of Mathematics, Univ. of Bologna, Italy\\
\email{mattia.bergomi@neuro.fchampalimaud.org, massimo.ferri@unibo.it, lorenzo.zuffi@studio.unibo.it},
}

\maketitle              % typeset the title of the contribution

\begin{abstract}
Graphs are a basic tool for the representation of modern data. 
The richness of the topological information contained in a graph goes far beyond its mere interpretation as a one-dimensional simplicial complex. We show how topological constructions can be used to gain information otherwise concealed by the low-dimensional nature of graphs. We do that by extending previous work of other researchers in homological persistence, by proposing novel graph-theoretical constructions. Beyond cliques, we use independent sets, neighborhoods, enclaveless sets and a Ramsey-inspired extended persistence.
\end{abstract}

\begin{keywords}
Clique, independent set, neighborhood, enclaveless set, Ramsey.
\end{keywords}
%\end{frontmatter}

\section{Introduction}

Currently data are produced massively and rapidly. A large part of these data is either naturally organized, or can be represented as  graphs or networks. In recent years, topological persistence has proved to be an invaluable tool for the exploration and understanding of these data types.

Albeit graphs can be considered as topological objects \textit{per se}, given their low dimensionality, only limited information can be obtained by studying their topology. It is possible to grasp more information by superposing higher-dimensional topological structures to a given graph. Many research lines, for example, focused on building $n$-dimensional complexes from graphs by considering their $(n+1)$-cliques (see, e.g., Section~\ref{sec:stateofart}).
Here, we explore how topological persistence can be used to study both classical and new simplicial complexes drawn from graphs. In particular, we will focus on the novel information disclosed by considering graph-theoretical concepts so far neglected in the literature, at least to our knowledge.

In Section~\ref{general} we set basic terminology and notation for simplicial complexes and graphs. Section~\ref{togrape} is devoted to the definition of several constructions of simplicial complexes based on as many basic graph-theoretical concepts, namely: The simplicial complexes of cliques, neighborhoods, enclaveless and independent sets. For each construction, we discuss both the simplicial complex representation and its filtration. The stability with respect to the bottleneck distance is examined in a separated subsection. Finally, we present an ``extended persistence''-like construction based on the Ramsey principle.

\subsection{State of art\label{sec:stateofart}}

Graphs and persistence are bound together since the early days, long before the term ``persistence'' was even coined \cite{VeUrFrFe93}. Graphs were a tool for managing the discretization of filtered spaces in applied contexts.

As far as we know, a true use of persistence  in the study of graphs {\it per se} started with \cite{HoMa*09}, where the clique and neighborhood complexes were built on a time-varying network for application in statistical mechanics (see also \cite{MaZh*16}). A research aiming at the physical application of persistence to polymer models of hypergraphs is developed in \cite{AlCo*17}.

Complex networks have been studied with persistent homology also in \cite{HuRi17,PaMo*17}. In both cases, the main example is a network of collaborating people; simplices are formed on the basis of relationship measures between members.  

Brain connections have been studied through complexes associated to graphs by various authors with different viewpoints and techniques, with exciting results: \cite{PeEx*14,ReNo*17,SiGi*18}.

\section{General preliminaries}\label{general}

We fix terminology for simplicial complexes and graphs respectively in Subsections~\ref{complexes} and \ref{graphs}.

\subsection{Simplicial complexes}\label{complexes}

First, we recall that a {\it simplicial complex} $K$ (an {\it abstract} simplicial complex in the terminology of many authors) is a set of simplices, where a {\it simplex} is a finite set of elements ({\it vertices}) of a given set $V(K)$, such that
\begin{enumerate}[(i)]
\item any set of exactly one vertex is a simplex of $K$
\item any nonempty subset of a simplex of $K$ is a simplex of $K$ \cite[Sect. 3.1]{Sp94}. 
\end{enumerate}
A simplex consisting of $n+1$ vertices is said to have {\it dimension} $n$ and to be an $n-$simplex. The dimension of $K$ is the maximum dimension of its simplices. In the remainder all simplicial complexes will be finite.

A standard way (actually a functor) to associate a topological space $|K|$ (the {\em space} of $K$) to a simplicial complex $K$ is through barycentric coordinates \cite[Sect. 3.1]{Sp94}. $|K|$ is defined as the set of functions $\alpha:K \to [0, 1]$ such that:
\begin{itemize}
\item For any $\alpha$, the set $\{v\in K \, | \, \alpha(v)\ne 0\}$ is a simplex of $K$,
\item For any $\alpha$, one has $\sum_{v\in K}\alpha(v)=1$
\end{itemize}

\noindent and the topology comes from the $L^2$ metric, but the most usual way of thinking of $K$ in geometrical terms is by its possible embeddings into a Euclidean space \cite[Sect. 3.2, Thm. 9]{Sp94}. This is what we shall do in both text and figures throughout the article.

We refer to \cite{Sp94} for terminology and notions of simplicial and algebraic topology; another very nice reference is \cite{Ha02}. For the sake of clarity, the simplex having vertices $v_0, \ldots, v_n$ will be denoted by $\langle v_0, \ldots, v_n \rangle$.

\subsection{Persistent homology}

\label{ph}
Persistent homology is a branch of computational topology, of remarkable success in shape analysis and pattern recognition. Its key idea is to analyze data through {\em filtering functions}, i.e. continuous functions $f$ defined on a suitable topological space $X$ with values e.g. in $\R$ (but sometimes in $\R^n$ or in a circle). Given a pair $(X, f)$, with $f:X \to \R$ continuous, for each $u\in \R$ the {\em sublevel set} $X_u$ is the set of elements of $X$ whose value through $f$ is less than or equal to $u$.

For each $X_u$ one can compute the {\em homology modules} $H_r(X_u)$. As there exist various homology theories, some additional hypotheses might be requested on $f$ depending on the choice of the homology. Here coefficients will be in a fixed field.

Of course, if $u<v$ then $X_u \subseteq X_v$. There corresponds a linear map $\iota^r_{(u, v)}: H_r(X_u)\to H_r(X_v)$. On $\Delta^+ = \{(u, v)\in \R^2 \, | \, u<v\}$ we can then define the {\em $r$-Persistent Betti Number} ($r$-PBN) function
\[\begin{array}{cccc}
   \beta^r_{(X,f)}: &\Delta^+ & \to & \mathbb{Z}\\
                    & (u, v) & \mapsto &\dim\textnormal{Im}(\iota^r_{(u, v)})
\end{array}\]
All information carried by $r$-PBN's is condensed in some points (dubbed {\em proper cornerpoints}) and some half-lines ({\em cornerlines}); cornerlines are actually thought of as {\em cornerpoints at infinity}. Cornerpoints (proper and at infinity) build what is called the {\em persistence diagram} relative to dimension $r$. Figure~\ref{M} shows a letter ``M'' as space $X$, ordinate as function $f$ on the left, its 0-PBN function at the center and the corresponding persistence diagram on the right.

\begin{figure}[htb]
\begin{center}
  \includegraphics[width=0.8\textwidth]{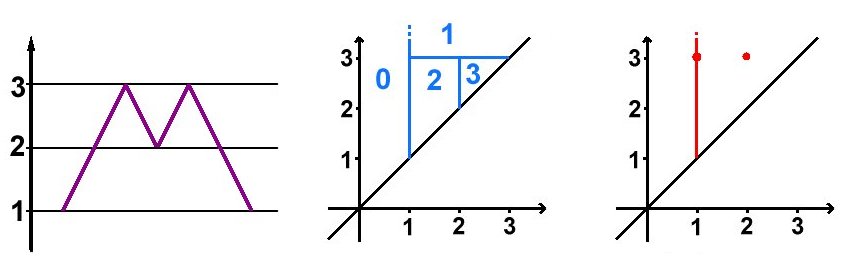}
    \caption{Letter M, its 0-PBM function and the corresponding persistence diagram, relative to filtering function ordinate.}\label{M}
    \end{center}
\end{figure}

\begin{remark}
The theory also contemplates a {\em multiplicity} for cornerpoints (proper and at infinity); multiplicity higher than one is generally due to symmetries.
\end{remark}

For homology theory one can consult any text on algebraic topology, e.g. \cite{Ha02}. For persistent homology, two good references are \cite{EdHa08,EdHa09}.

\subsection{Graphs}\label{graphs}

For the purposes of this article, a graph $G$ will be defined as a simplicial complex of dimension $1$; we shall write $G=\big(V(G), E(G)\big)$ where $V$ is its set of 0-simplices or {\em vertices} and $E$ its set of 1-simplices or {\em edges}.
So, in graph-theoretical terms they are finite simple graphs. The category {\bf Graph} will have graphs as objects and simplicial maps as morphisms. For both graph-theoretical notions and terminology we refer to \cite{BoMu11}.

A {\em weighted graph} will be a pair $(G, f)$ where $G=\big(V(G),E(G)\big)$ is a graph and $f:E \to \mathbb{R}$ is a function, called {\em weight function}; sometimes the range of $f$ will instead be $\mathbb{R} \cup \{+\infty\}$.

\section{Graph persistence}\label{togrape}

\begin{figure}[htb]
\centering
\begin{minipage}[ht]{0.3\linewidth}
\centering
\includegraphics[width=1\linewidth]{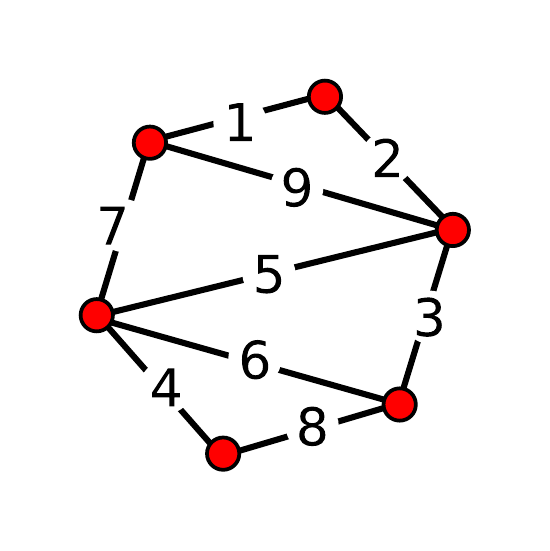}
\end{minipage}
\hfill
\begin{minipage}[ht]{0.33\linewidth}
\centering
\includegraphics[width=1\linewidth]{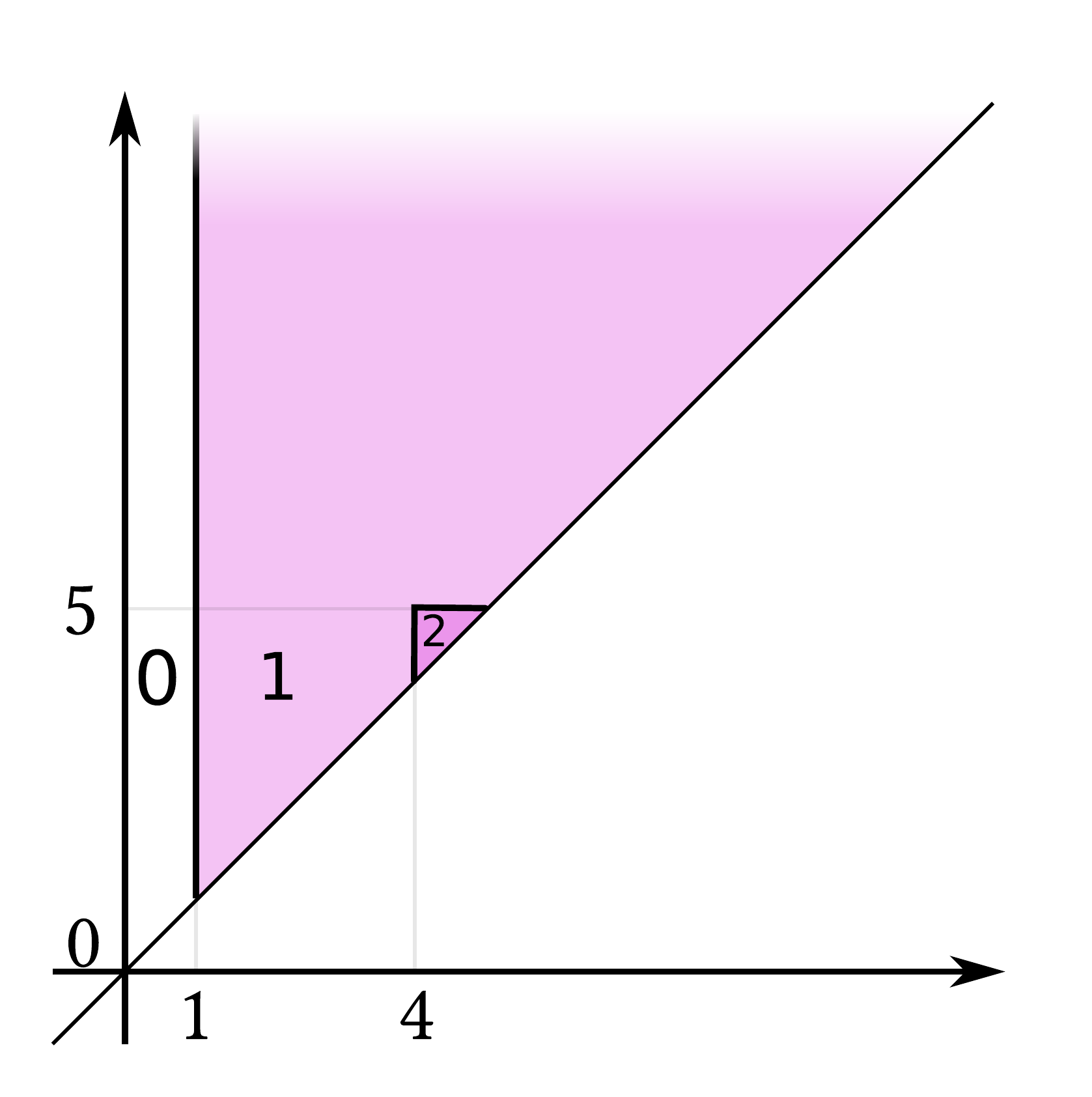}
\end{minipage}
\hfill
\begin{minipage}[ht]{0.33\linewidth}
\centering
\includegraphics[width=1\linewidth]{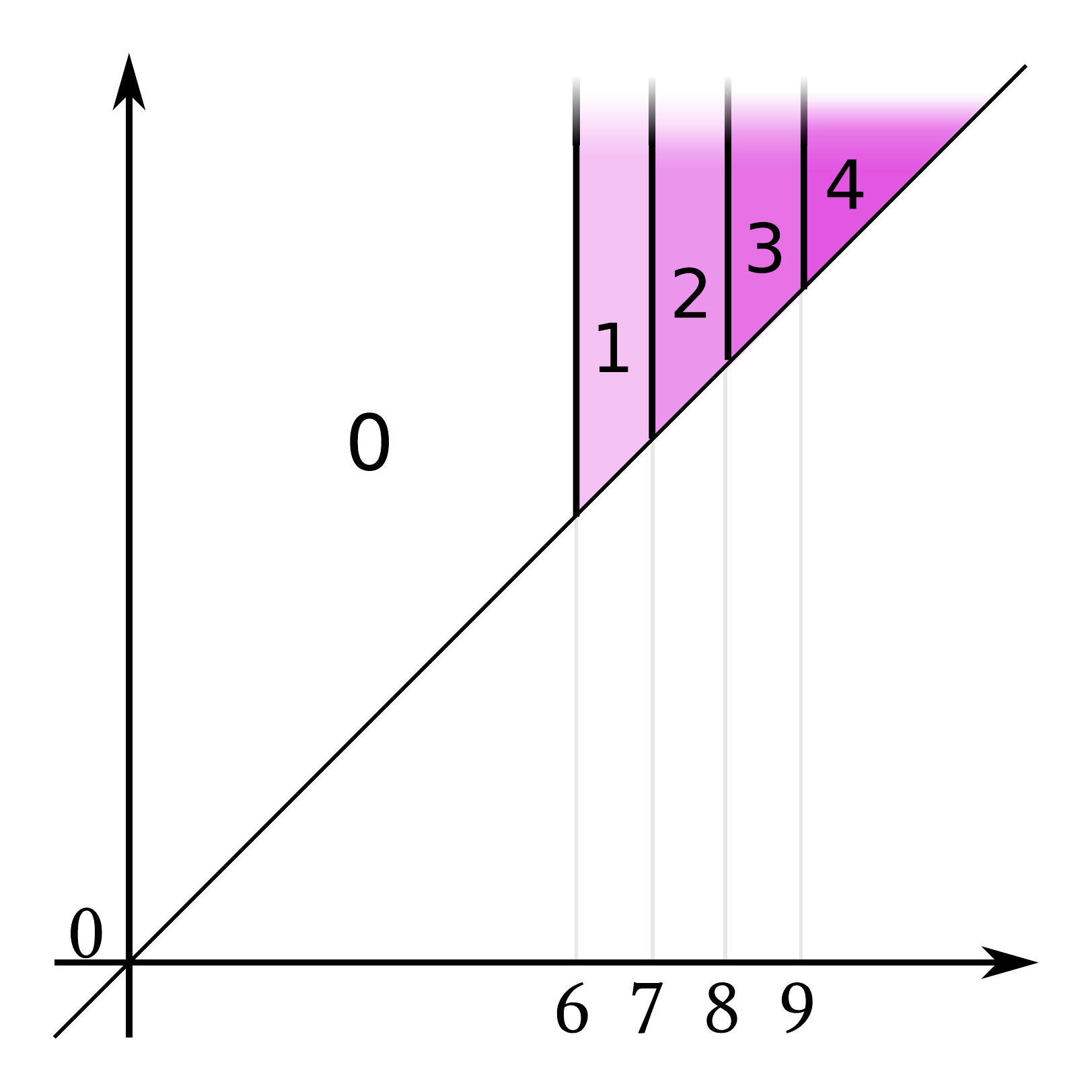}
\end{minipage}
\caption{From left to right: A weighted graph, its 0-PBN and its 1-PBN functions as the space of a simplicial complex.}\label{fig:betti}
\end{figure}

\noindent Of course, we can obtain a filtered topological space as the space of the graph. This gives rise to 0- and 1-PBN functions and persistence diagrams (see, e.g., Fig.~\ref{fig:betti}). In our opinion, this is too limited a view on the graph. More and more interesting information, on the relations represented by the weighted graph, can be conveyed by building other simplicial complexes related with it.

The ``leit-motiv'' of this section is the analysis of simplicial complexes built from (weighted) graphs, to study then the filtration of the complexes with the methods of persistent homology. Our goal is actually to spot particular classes of sets in a graph such that the conditions (i) and (ii) of the definition of simplicial complex hold. A thorough treatise on these configurations can be found in \cite{Jo08}. 

\subsection{Complex of cliques}\label{sec:cli}

Let $G$ be a graph. A $k$-\textit{clique} of $G$ is a set of $k$ vertices ($k>0$) whose induced subgraph is complete. We recall that the set $Cl(G)$ of cliques in $G$ is a simplicial complex. 

In general, not every simplicial complex $K$ can be represented as $Cl\left(G\right)$ for some $G$: Let $K$ be the simplicial complex formed by an $h$-simplex $\sigma$ ($h>1$) and all its faces. $K$ and $K-\{\sigma\}$ share the same 1-skeleton $K^1$: There does not exist a graph $G^\prime$ such that $Cl\left(G^\prime\right) = K-\{\sigma\}$. 

It is possible to overcome this issue by considering the barycentric subdivision of the simplicial complex. Formally, for every simplicial complex $K$, let $K^\prime$ be its barycentric subdivision and  $G = \left(K^\prime\right)^1$ the graph built on the $1$-skeleton of $K^\prime$. Then, we have that $Cl\left(G\right)$ is isomorphic to $K^\prime$ and $\left|Cl\left(G\right)\right|$ is homeomorphic to $|K|$.

If $K= Cl(G)$, its suspension $\Sigma(K)$ is the  complex of cliques of
$$CSusp(G) = \big(V(G)\cup \{x, y\}, E(G) \cup \{\langle x, v \rangle, \langle v, y\rangle \, | \, v\in V(G)\}\big)$$
where $x, y \not\in V(G)$.
In particular, any sphere of dimension $h\ge 1$ can be triangulated by the clique complex of a suitable nonempty graph, e.g. starting from a 4-cycle for $\mathbb{S}^1$ and applying $CSusp$ the necessary number of times. Therefore, the following corollary holds.

\begin{corollary}\label{clbetti}
For any finite sequence $\sigma$ of nonnegative integers, there exists a graph $G$ such that $\sigma$ is the sequence of Betti numbers of $Cl(G)$.
\hfill $\square$
\end{corollary}

Finally, with the two following proposition we prove how a filtration of simplicial complexes can be associated with a filtration of graphs.

\begin{proposition}\label{clmonot}
If $G$ is a subgraph of $H$, $Cl(G)$ is a subcomplex of $Cl(H)$.
\end{proposition}
\begin{proof}
Every clique of $G$ is also a clique of $H$.
\hfill $\square$
\end{proof}

Let now $(G, f)$ be a weighted graph. We define a filtering function $f_{Cl}: Cl(G) \to \mathbb{R}$ as follows:
\begin{itemize}
\item for every 0-simplex $\sigma=\langle v \rangle$, $f_{Cl}(\sigma)$ is the minimum value of $f$ on the edges incident on $v$;
\item for every $k$-simplex $\sigma$ ($k\ge 1$), i.e. for every $(k+1)$-clique, $f_{Cl}(\sigma)$ is the maximum value of $f$ on the edges of the induced complete subgraph.
\end{itemize}

\begin{figure}[htb]
\centering
\begin{minipage}[ht]{0.3\linewidth}
\centering
\includegraphics[width=1\linewidth]{filtrazione7.pdf}
\end{minipage}
\hfill
\begin{minipage}[ht]{0.33\linewidth}
\centering
\includegraphics[width=1\linewidth]{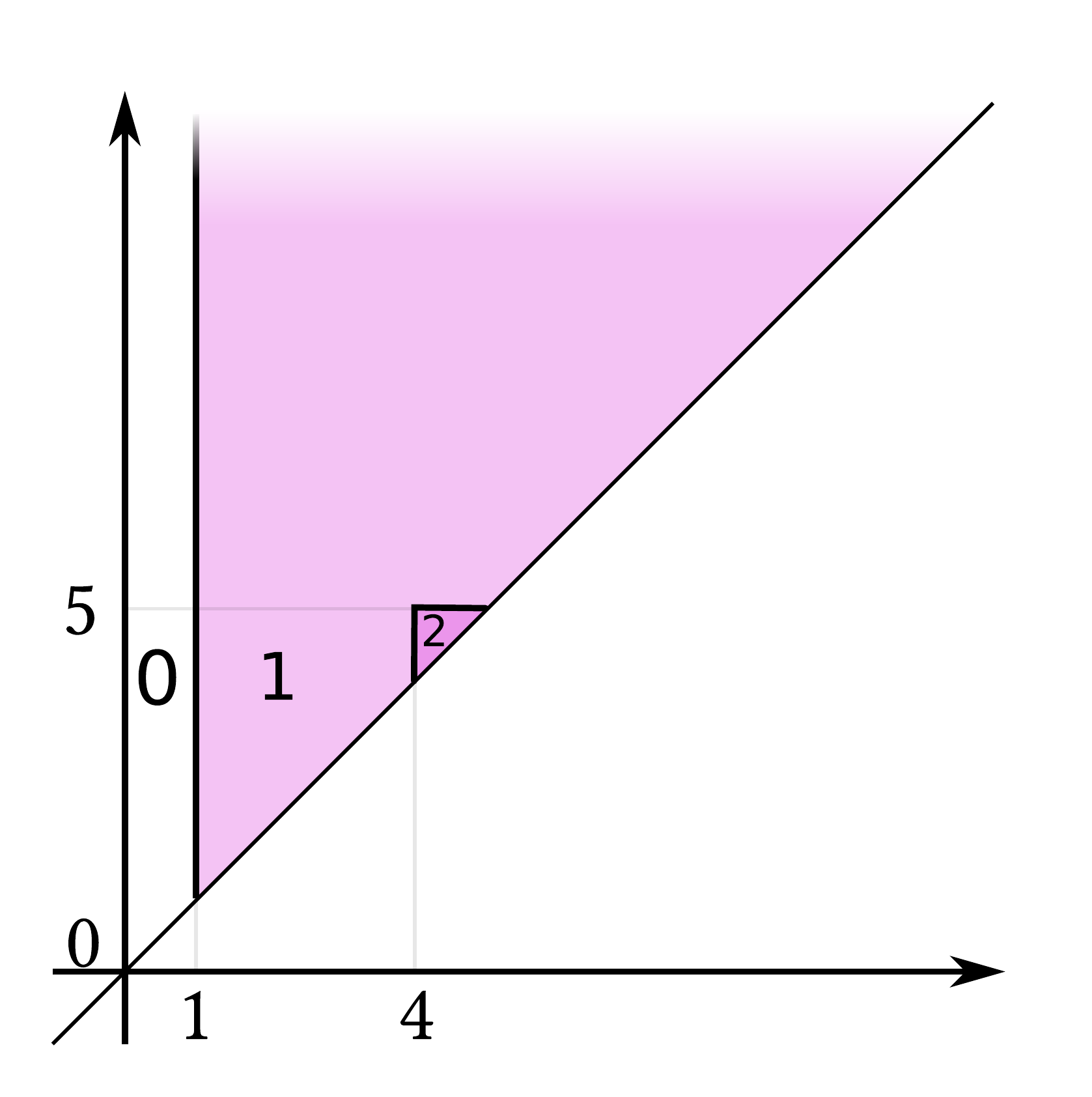}
\end{minipage}
\hfill
\begin{minipage}[ht]{0.33\linewidth}
\centering
\includegraphics[width=1\linewidth]{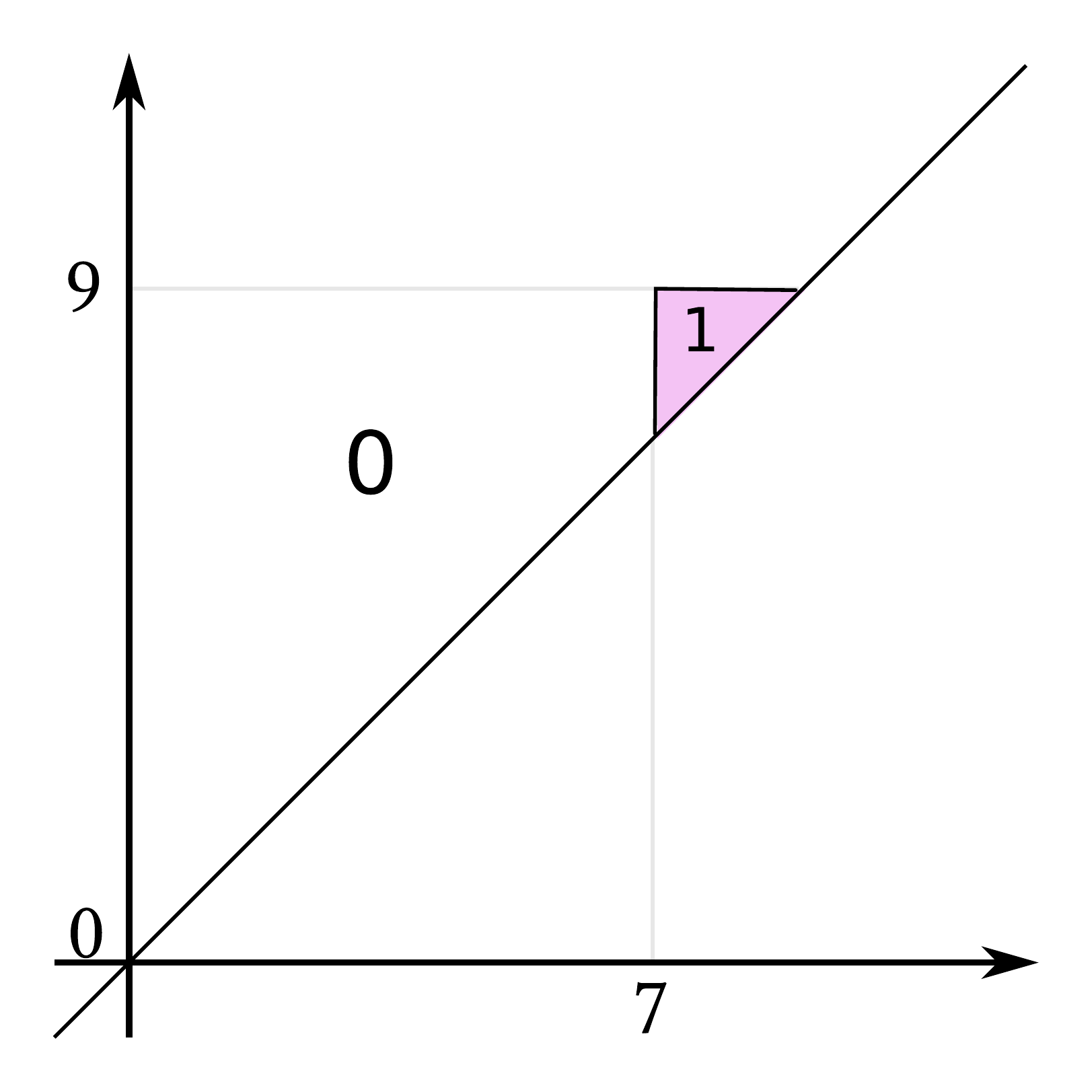}
\end{minipage}
\caption{From left to right: The same weighted graph, the 0-PBN and its 1-PBN functions of its clique complex.}\label{fig:clbetti}
\end{figure}

\begin{proposition}
$\big(Cl(G), f_{Cl}\big)$ is a filtered complex.
\end{proposition}
\begin{proof}
By construction, the value of every simplex is $\ge$ the value of each of its faces.
\hfill $\square$
\end{proof}

Fig.~\ref{fig:clbetti} shows the PBN's of the filtered clique complex of the already seen weighted graph of Fig.~\ref{fig:betti}. Of course, the conveyed information is totally different.

\subsection{Complex of neighborhoods}\label{neigh}

In a graph $G=(V, E)$, given $v\in V$, its {\em neighborhood} in $G$ is the set $N_G(v)=\{v\}\cup \{u\in V \, | \, \langle v, u \rangle \in E\}$. Given a graph $G$, the set $Nb(G)$ of all nonempty subsets of neighborhoods of vertices of $G$ is a simplicial complex. We refer the reader to \cite{Lo78,MaRa*08} for the proof of this claim.

Not all simplicial complexes can be obtained by considering the complex of neighborhoods of a graph. Consider, for instance, the boundary of a triangle. Furthermore, the barycentric subdivision strategy---that has proved to be successful in Sections.~\ref{sec:cli} and~\ref{sec:ind}---does not work in the neighborhoods' framework. Despite these issues, we repute the complex of neighborhoods to be a valuable construction: By definition $Nb(G)$ diverges greatly from the topology of $G$ as a simplicial complex, thereby revealing novel information about the combinatorics of $G$. Fig.~\ref{ntop} shows how trivially homeomorphic cycles give rise to nonhomotopic complexes of neighborhoods. Moreover, as stressed in \cite{MaRa*08,HoMa*09}, this construction seems to be a precious tool for the analysis of complex networks.

\begin{figure}[tb]
\centering
\includegraphics[scale=0.4]{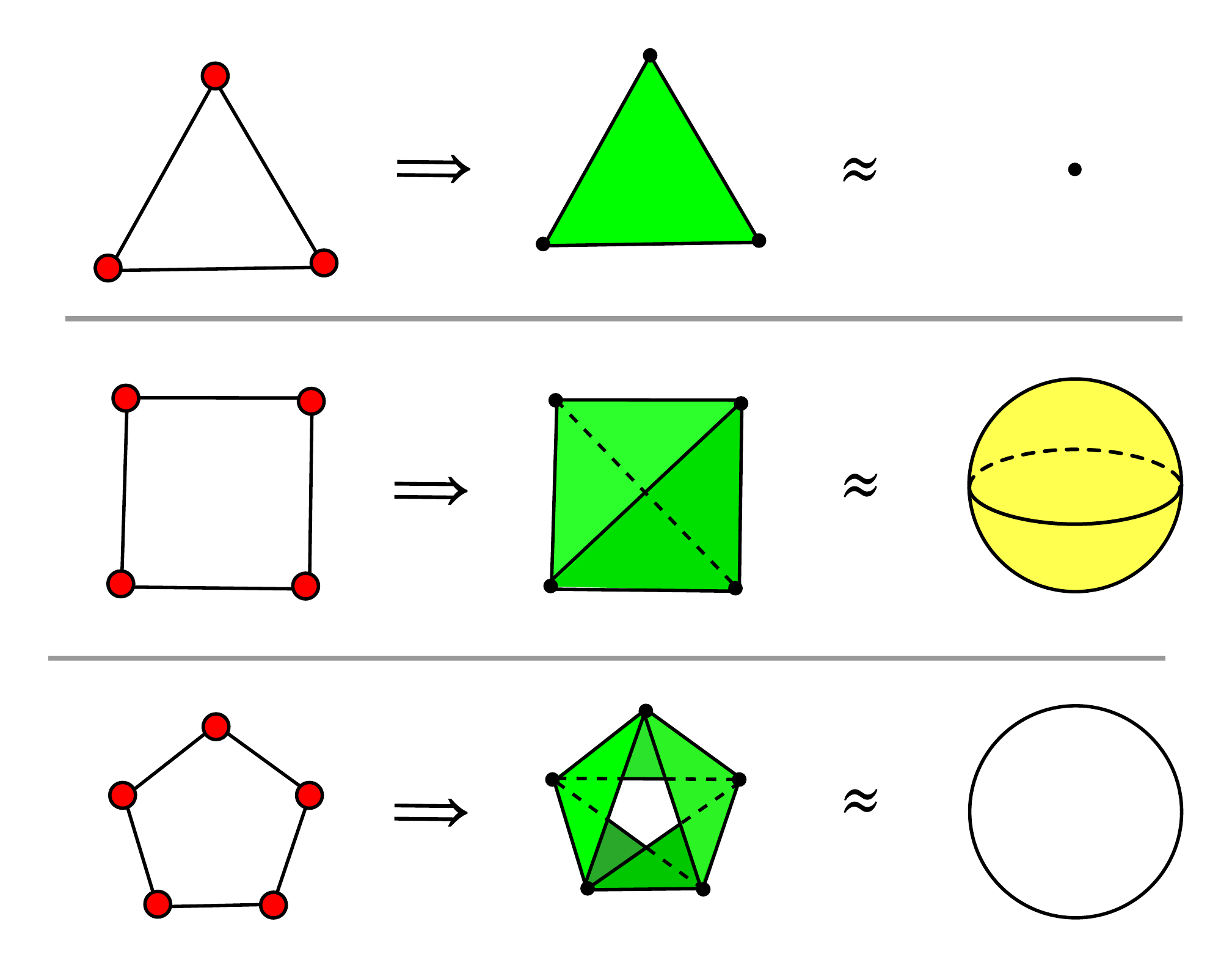}
\caption{Nonhomotopic complexes of neighborhoods of three cycles.}\label{ntop}
\end{figure}

\begin{proposition}\label{nbmonot}
If $G$ is a subgraph of $H$, $Nb(G)$ is a subcomplex of $Nb(H)$.
\end{proposition}
\begin{proof}
The neighborhood of every vertex $v$ of $G$ is a subset of the neighborhood of $v$ in $H$.
\hfill $\square$
\end{proof}

Let now $(G, f)$ be a weighted graph. We define a filtering function $f_{Nb}: Nb(G) \to \mathbb{R}$ as follows:
\begin{itemize}
\item for every 0-simplex $\sigma=\langle v \rangle$, $f_{Nb}(\sigma)$ is the minimum value of $f$ on the edges incident on $v$;
\item calling $G_t$ the subgraph induced by all edges $e$ with $f(e) \le t$, for a $k$-simplex $\sigma$, $k\ge 1$, we set $f_{Nb}(\sigma)=\overline{t}$, where $\overline{t}$ is the smallest value of $t$ for which $\sigma$ is a subset of a neighborhood in $G_t$.
\end{itemize}

% \begin{figure}[htb]
% \centering
% \begin{minipage}[ht]{0.35\linewidth}
% \centering
% \includegraphics[width=1\linewidth]{filtrazione7.pdf}
% \end{minipage}
% \hfill
% %
% \begin{minipage}[ht]{0.4\linewidth}
% \centering
% \includegraphics[width=1\linewidth]{neighbours_beta0_function.pdf}
% \end{minipage}
% \hfill
% %
% \caption{From left to right: The weighted graph and the 0-PBN of its neighborhood complex (1-PBN is trivial throughout.}\label{fig:nbbetti}
% \end{figure}

\begin{proposition}
$\big(Nb(G), f_{Nb}\big)$ is a filtered complex.
\end{proposition}
\begin{proof}
By construction, the value of every simplex is greater than or equal to the value of each of its faces.
\hfill $\square$
\end{proof}

% Fig.~\ref{fig:nbbetti} shows the PBN's of the filtered neighborhood complex of the same weighted graph of Fig.~\ref{fig:betti}.

As an example, the neighborhood complex of the usual weighted graph of Figs.~\ref{fig:betti} and \ref{fig:clbetti} has the same 0-PBN function as the clique complex, but trivial 1-PBN.

\subsection{Complex of enclaveless sets}

Separation is as important as closeness in a network, and this is well represented by the duality clique/independent set. Somehow, there is a concept that merges these two aspects: In a graph $G=(V, E)$ a set $X\subseteq V$ is said to be {\em dominating} if every vertex of $G$ belongs either to $X$ or is adjacent to at least one of its vertices. 

Unfortunately, the inheritance property (ii) of simplicial complexes does not hold for dominating sets; on the contrary, every superset of a dominating set is dominating. So we turn to their complementary sets. A set $Y\subseteq V$ is said to be {\em enclaveless} if for no $v\in Y$ we have $N(v)\subseteq Y$. We observe that, given a graph $G$, the set $El(G)$ of all its enclaveless sets is a simplicial complex.

\begin{figure}[tb]
\centering
\includegraphics[scale=0.7]{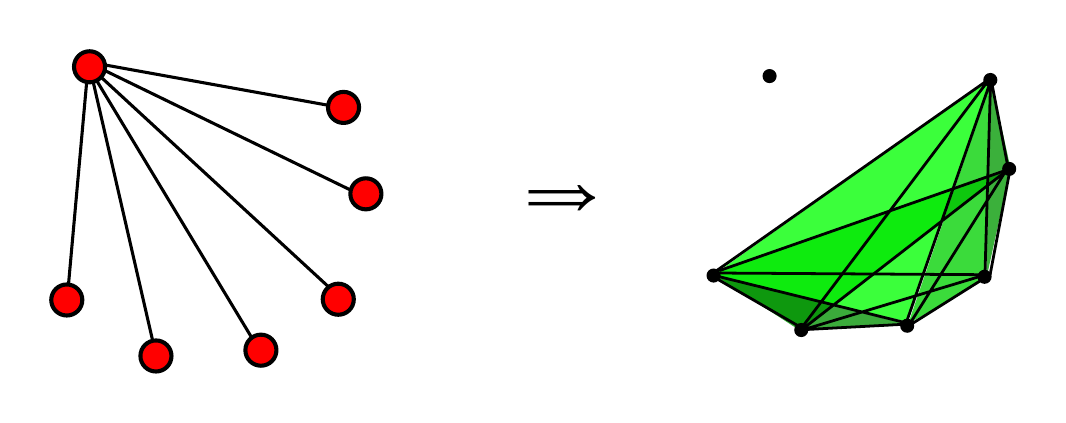}
\caption{The smallest graph containing an enclaveless set of cardinality 6 and the corresponding complex.}\label{elsimp}
\end{figure}

Not all simplicial complexes can be obtained as complex of enclaveless sets of a graph: e.g. the one formed by a single $n$-simplex and its faces; the smallest $G$ such that $El\left(G\right)$ contains an $n$-simplex must also contain an extra point. See Fig.~\ref{elsimp} for $n=5$. Still, spheres of any dimension and suspensions can be obtained.

If $K= El\left(G\right)$, then its suspension $\Sigma(K)$ is the complex of enclaveless sets of
$$ElSusp(G) = \big(V(G)\cup \{x, y\}, E(G) \cup \{\langle x, y\rangle \}\big)$$
where $x, y \not\in V(G)$. It turns out to be the same construction as $ISusp$ for the complex of independent sets (see Sect.~\ref{sec:ind}).
Spheres can be built also in another way:

\begin{proposition}
For $n\ge 2$ the space of $El(K_n)$ is homeomorphic to $\mathbb{S}^{n-2}$
\end{proposition}
\begin{proof}
Minimal dominant sets in a complete graph $K_n$ are all singletons; so the maximal enclaveless sets are all sets of $n-1$ vertices; they together form the boundary of an $(n-1)$-simplex, {while the $(n-1)$-simplex itself} is not present in $El(K_n)$. 
\hfill $\square$
\end{proof}

\begin{proposition}\label{elmonot}
If $G$ is a subgraph of $H$, $El(G)$ is a subcomplex of $El(H)$.
\end{proposition}
\begin{proof}
Every dominating set of $H$ is also a dominating set of $G$.
\hfill $\square$
\end{proof}

Let now $\big(G=(V, E), f\big)$ be a weighted graph. We define a filtering function $f_{El}: El(G) \to \mathbb{R}$ as follows:
\begin{itemize}
\item for every 0-simplex $\sigma=\langle v \rangle$, $f_{El}(\sigma)$ is the minimum value of $f$ on the edges incident on $v$;
\item calling $G_t$ the subgraph induced by all edges $e$ with $f(e) \le t$, for a $k$-simplex $\sigma$, $k\ge 1$, we set $f_{El}(\sigma)=\overline{t}$, where $\overline{t}$ is the smallest value of $t$ for which $\sigma$ is a subset of an enclaveless set in $G_t$.
\end{itemize}

\begin{proposition}
$\big(El(G), f_{El}\big)$ is a filtered complex.
\end{proposition}
\begin{proof}
By construction, the value of every simplex is greater than or equal to the value of each of its faces.
\hfill $\square$
\end{proof}

\begin{figure}[htb]
\centering
\begin{minipage}[ht]{0.17\linewidth}
\centering
\includegraphics[width=1\linewidth]{filtrazione7.pdf}
\end{minipage}
\hfill
\begin{minipage}[ht]{0.26\linewidth}
\centering
\includegraphics[width=1\linewidth]{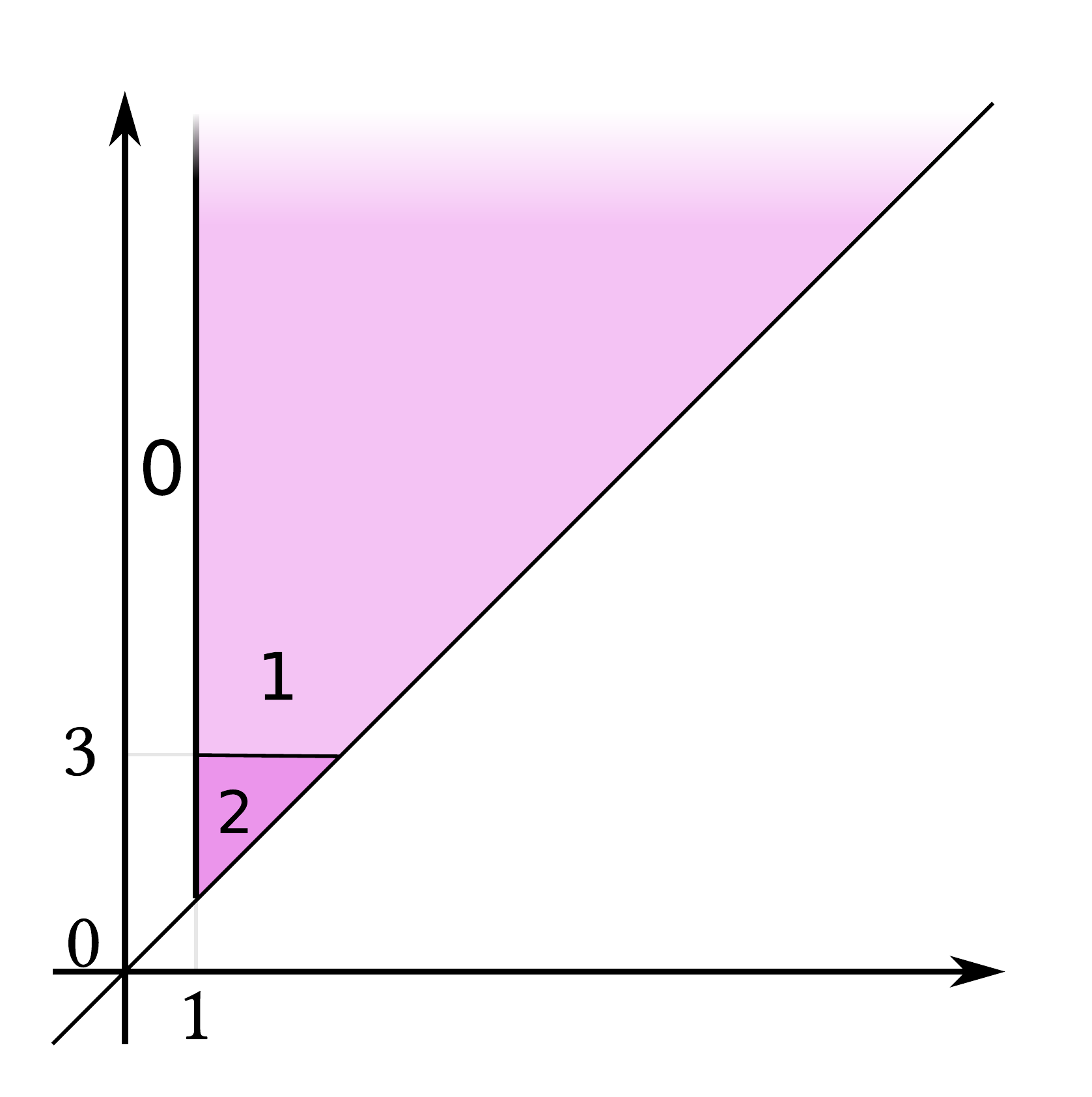}
\end{minipage}
\hfill
\begin{minipage}[ht]{0.26\linewidth}
\centering
\includegraphics[width=1\linewidth]{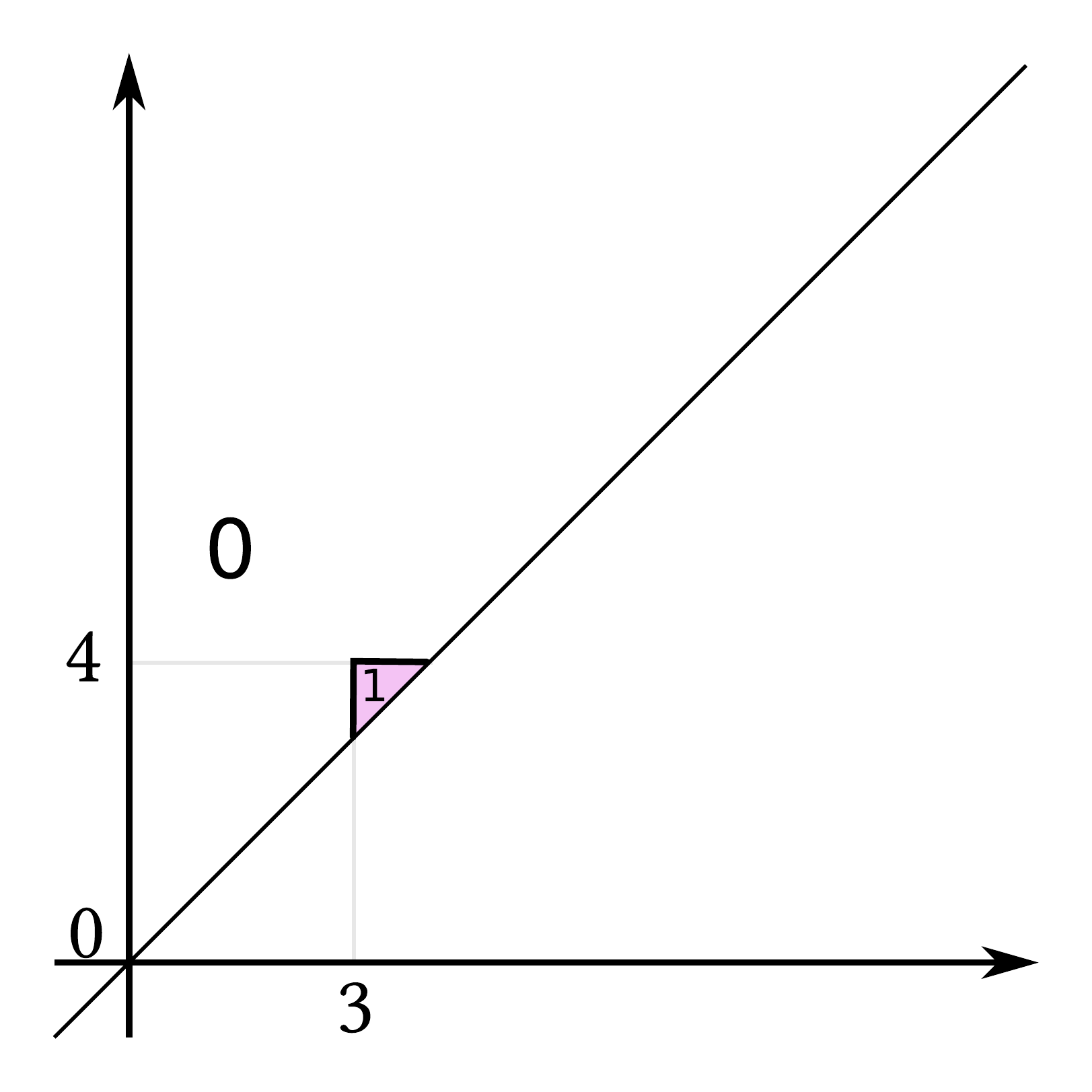}
\end{minipage}
\hfill
\begin{minipage}[ht]{0.26\linewidth}
\centering
\includegraphics[width=1\linewidth]{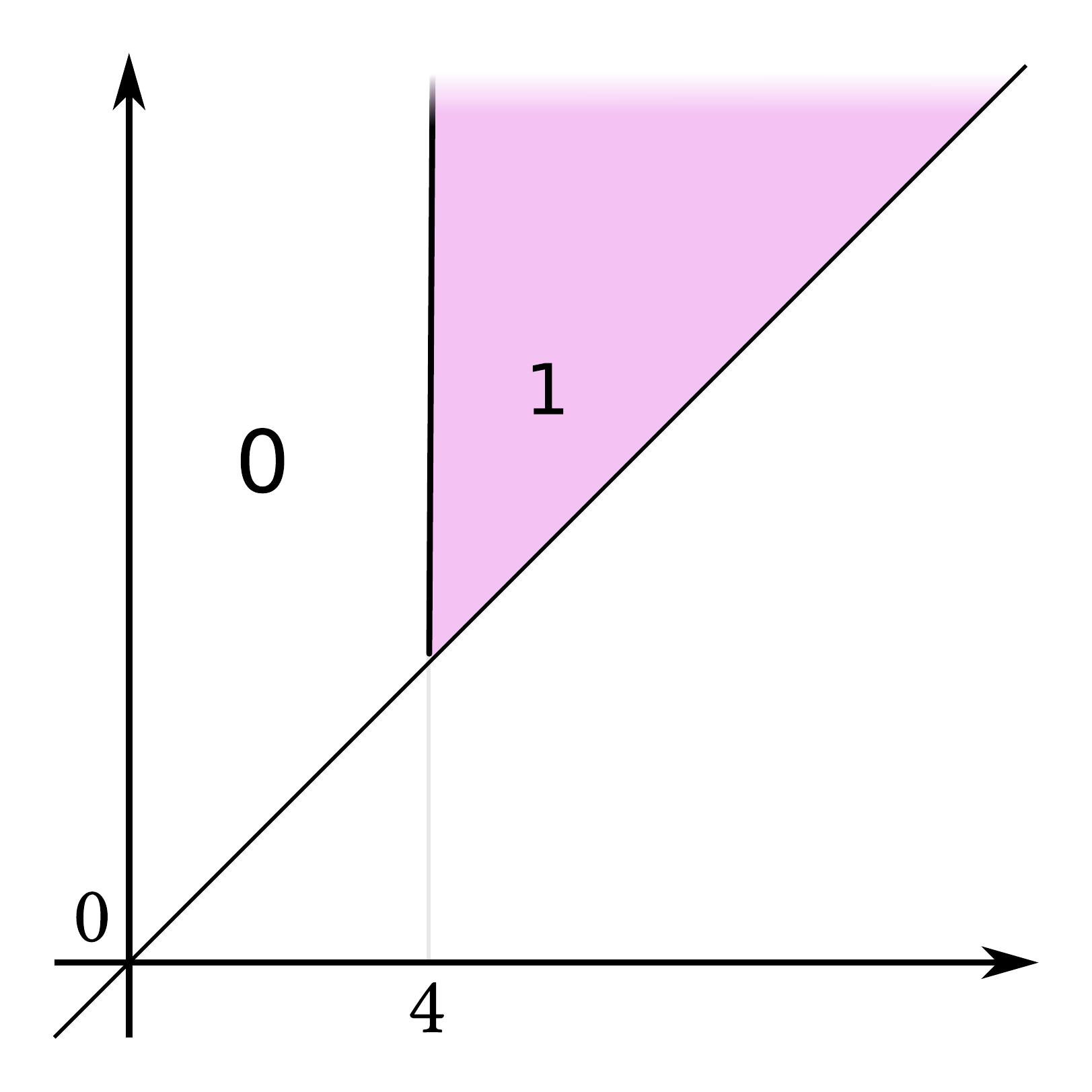}
\end{minipage}
\caption{From left to right: The same weighted graph, the 0-PBN, the 1-PBN and the 2-PBN functions of its complex of enclaveless sets.}\label{fig:enbetti}

\end{figure}

Fig.~\ref{fig:enbetti} shows the PBN functions for the complex of enclaveless sets of the weighted graph of Figs.~\ref{fig:betti} and \ref{fig:clbetti}. Here we have even a non trivial function in dimension 2.

\subsection{Complex of independent sets}\label{sec:ind}

An {\em independent} (or {\em stable}) nonempty set in a graph is a set of vertices such that the induced subgraph does not contain any edge. Recall that, given a graph $G=(V, E)$, its {\em complement} is the graph $G^c = \left(V, E'\right)$ where for all $u, v \in V, \ u\neq v$, $\langle u, v \rangle \in E'$ if and only if $\langle u, v \rangle \not\in E$; i.e. it has the same vertex set as $G$ and its edge set is complementary to $E$ with respect to the complete graph with the same vertices. Then a set of vertices is independent in $G$ if and only if it is a clique in $G^c$ and conversely. 

Even in this case, we observe that given a graph $G$, the set $I(G)$ of its independent sets is a simplicial complex. Again, not every simplicial complex can be described as the complex of independent sets of a graph. The same barycentric subdivision strategy also applies in this framework. For any simplicial complex $K$, let $K'$ be its barycentric subdivision. Then for the graph $G=(K')^1$ (1-skeleton of $K'$) we have that $I(G^c)$ is isomorphic to $K'$ and $|I(G^c)|$ is homeomorphic to $|K|$.

If $K= I(G)$, then its suspension $\Sigma(K)$ is the complex of independent sets of
$$ISusp(G) = \big(V(G)\cup \{x, y\}, E(G) \cup \{\langle x, y\rangle \}\big)$$
where $x, y \not\in V(G)$. I.e. $ISusp(G)$ is the graph obtained from $G$, by adding a component formed by a single edge. Note that $ISusp(G)$ is \textit{not} the suspension of $G$.
In addition, a sphere of any dimension can be triangulated by a suitable $I(G)$. Thus, For any finite sequence $\sigma$ of nonnegative integers, there exists a graph $G$ such that $\sigma$ is the sequence of Betti numbers of $I(G)$.

The sort of duality between cliques and independent sets implies that the monotonically increasing correspondence of Prop.~\ref{clmonot} becomes decreasing here:

\begin{proposition}\label{imonot}
If $G$ is a subgraph of $H$, $I(H)$ is a subcomplex of $I(G)$.
\hfill $\square$
\end{proposition}

This makes it impossible to associate a filtered complex to a weighted graph in the same way as in the previous sections. Still, we shall treat this case together with the clique complex in Section~\ref{ramsey}.

\subsection{Other complexes from graphs}

There are several other classes of sets in a graph $G=(V, E)$ which respect the definition of simplicial complex \cite{Jo08}. We have done a preliminary study on the following ones.

\begin{itemize}
\item the nonempty sets $\sigma\subseteq V$ such that the subgraph induced by $\sigma$ is acyclic;
\item the nonempty sets $\sigma\subseteq E$ such that the subgraph induced by $\sigma$ is acyclic;
\item (with $G$ connected) the nonempty sets $\sigma\subseteq E$ such that the subgraph induced by $\sigma$ is acyclic and the subgraph induced by $E-\sigma$ is connected.
\item for fixed positive $j<k$, the sets of $k$-cycles which share at least $j$ vertices;
\item for fixed positive $l$, the sets of maximal cliques sharing at least $l$ vertices.
\end{itemize}

We have decided to postpone the study of these complexes: The first three because non-isomorphic subgraphs may induce simplices of the same dimensions, so much structure is forgotten. As for the last two, the dependence on $j, k$ and $l$ respectively suggests that this type of complex might be of use in very specific applications.

\subsection{Distances}
% In the previous section we showed how, given a weighted graph $(G, f)$, it is possible to compute persistence diagrams representing the evolution of various graph-theoretical properties of $G$, along the filtration.

We recall that, for any two pairs $(X, f)$, $(Y, g)$, with $X, Y$ topological spaces and $ f: X \to \R^n$, $ g: Y
\to \R^n$ continuous, we have the following definition (see, e.g., \cite{dAFrLa06}).

\begin{definition}
The \emph{natural pseudodistance} between the pairs $(X,
f)$ and $(Y, g)$, denoted by
$\delta\left((X,  f),(Y,  g)\right)$,  is
\begin{itemize}
\item[(i)] the number $\inf_{h}\max _{x\in X}\|
f(x)- g(h(x))\|$ where $h$ varies in the set
$H(X,Y)$ of all the homeomorphisms between $X$ and $Y$, if $X$ and
$Y$ are homeomorphic;
\item[(ii)] $+\infty$, if $X$ and $Y$ are
not homeomorphic.
\end{itemize}
\end{definition}\label{pseudodistance}

Classification and retrieval of persistence diagrams (and consequently of the object they represent) is usually performed by the following distance, where persistence diagrams are completed by all points on the ``diagonal'' $\Delta=\{(u, v)\in\R \, | \, u=v\}$.

\begin{definition}
\emph{Bottleneck} (or \emph{matching}) \emph{distance}.
\\ Let $\mathcal{D}_{k}$ and $\mathcal{D'}_{k}$ be two persistence diagrams with a finite number of cornerpoints, the \emph{bottleneck distance} $d(\mathcal{D}_{k},\mathcal{D'}_{k})$ is defined as
\begin{equation}
d(\mathcal{D}_{k},\mathcal{D'}_{k})=\min_{\sigma}\max_{P\in\mathcal{D}_{k}}\hat{d}(P,\sigma(P))
\end{equation}
where $\sigma$ varies among all the bijections between $\mathcal{D}_{k}$ and $\mathcal{D'}_{k}$ and
\begin{equation}
\hat{d}((u,v),(u',v'))=\min\left\{\max\left\{|u-u'|,|v-v'|\right\},\max\left\{\frac{v-u}{2},\frac{v'-u'}{2}\right\}\right\}
\end{equation}
given $(u,v)\in\mathcal{D}_{k}$ and $(u',v')\in\mathcal{D'}_{k}$.
\end{definition}\label{match}

The next easy proposition does not provide any new lower bound, but assures us that small changes in the filtering function produce small changes in the persistence diagrams relative to the different constructions. 

\begin{proposition}
Let $(G, f)$, $(G', f')$ be weighted graphs and $D(f)$, $D(f')$ be the persistence diagrams of the persistent $r$-Betti numbers of $\big(K(G), f_K\big), \big(K(G', f'_K)\big)$ respectively, for $r$ fixed, $K = Cl, Nb, El$. Then
$$d\big(D(f), D(f')\big) \le \delta\big((G,  f), (G', f')\big)$$
\end{proposition}
\begin{proof}
Let the graph $G=(V, E)$ be isomorphic to $G'$; then $K(G)$ is a complex isomorphic to $K(G')$ (and the polyhedra $|K(G)|, \ |K(G')|$ are homeomorphic). To each isomorphism from $G$ to $G'$ there corresponds an isomorphism from $K(G)$ to $K(G')$ (and a homeomorphism from $|K(G)|$ to $|K(G')|$), but not conversely, in general. So

\medskip
%\newline
\null \hfill$\delta\Big(\big(K(G), f_K\big),  \big(K(G'), f'_K)\big)\Big) = \min_{\varphi\in \tilde{H}} \max_{\sigma\in K(G)}|f_K(\sigma)-f'_K\big(\varphi(\sigma)\big)| \le$\hfill\null
\newline
\null\hfill $\le\min_{\psi\in \overline{H}}\max_{e\in E}|f(e)-f'\big(\psi(e)\big)| = \delta\big((G,  f), (G', f')\big)$ \hfill\null
\medskip
\newline
where $\tilde{H}$ is the set of all simplicial isomorphisms from $K(G)$ to $K(G')$ and $\overline{H}$ is the set of all graph isomorphisms from $K(G)$ to $K(G')$. On the other side,
$$d\big(D(f), D(f')\big) \le \delta\Big(\big(K(G), f_K\big),  \big(K(G'), f'_K)\big)\Big)$$
is the classical stability result for filtered complexes or topological spaces (see, e.g., \cite{CoEdHa07,ChaCo*09,dAFrLa10,Les15}).
\hfill $\square$
\end{proof}

\begin{remark}
Unfortunately, the inequality between the natural pseudodistances of filtered complexes and filtered graphs may be strict, so we cannot get an universality result yet, although the elegant construction of \cite[Prop. 5.8]{Les15} applies to $K=Cl$ through Prop.~\ref{clbetti}.
\end{remark}

As hinted previously, $I(G)$ is excluded from this subsection, because of its monotonically decreasing behaviour with respect to inclusion (Prop.~\ref{imonot}). However, there is an interesting "extended" diagram that comes exactly from this phenomenon.

\subsection{An extended persistence of Ramsey type}\label{ramsey}

The celebrated Ramsey principle \cite{Ra30}, in its most common graph-theoretical version \cite[sect. 12.3]{BoMu11} makes it clear that one should also consider independent sets if one is interested in cliques and conversely. An equivalent viewpoint is: When interested in the information conveyed by the cliques of $G$, it is natural to also take into account the cliques of $G^c$. In the persistence field, the concept of {\em extended persistence} \cite{CoEdHa09} explores the lower half-plane $\Delta^- = \{(u, v)\in \mathbb{R} \, | \, u>v\}$ by relative homology.
We somehow merge these two philosophies in the following setting.

\begin{figure}[tb]
\centering
\includegraphics[width = \textwidth]{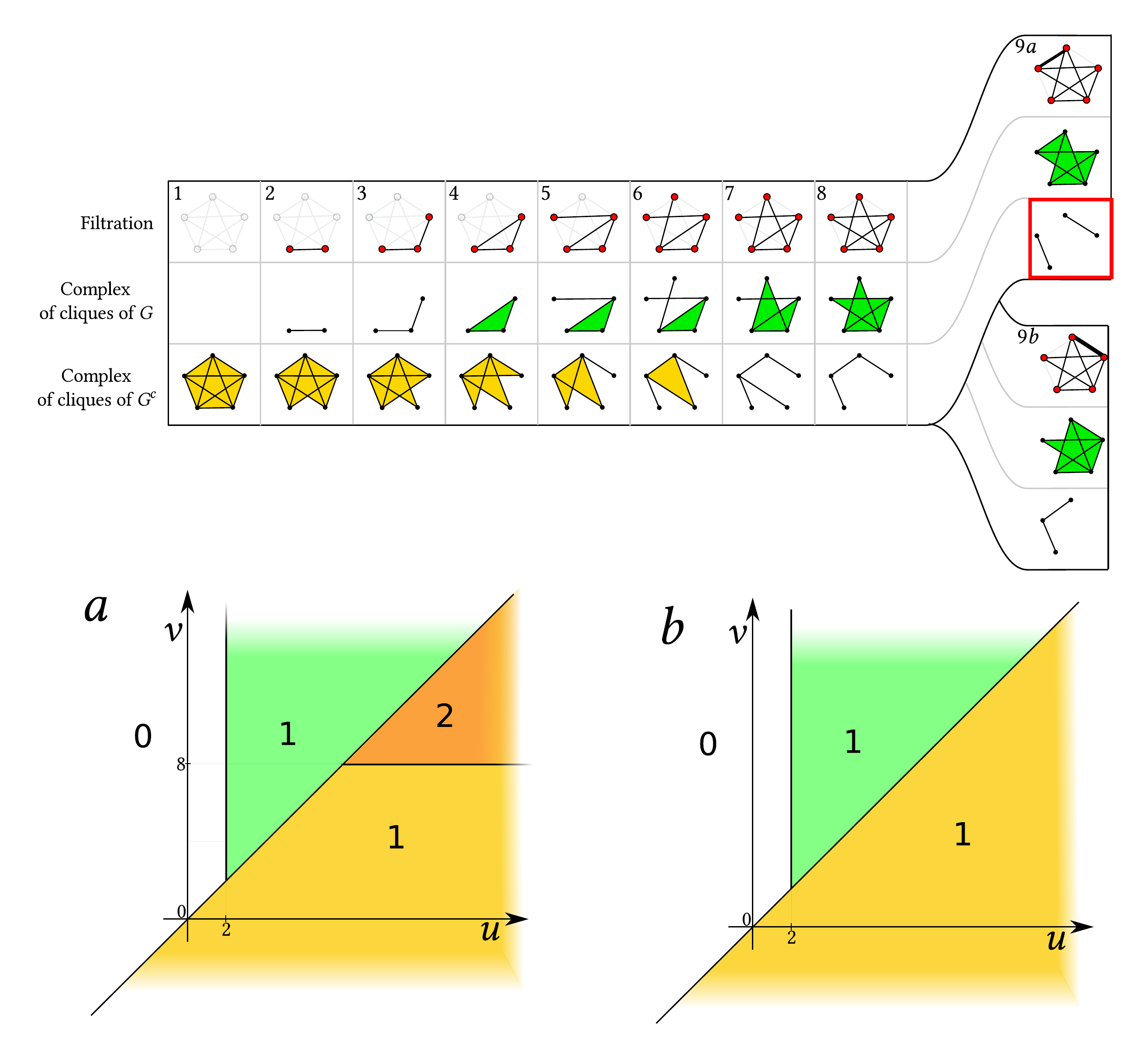}
\caption{Two filtrations and their extended persistent 0-Betti number functions.\label{extfig}}
\end{figure}

Given a weighted graph $\big(G=(V, E), f\big)$,
build the pair $(\overline{G}, \overline{f})$, where $\overline{G}=(V, \overline{E})$ is the complete graph on the vertex set $V$, and
$$\begin{array}{cccc}
\overline{f}:& \overline{E}&\to & \mathbb{R}\cup \{+\infty\}\\
&e &\mapsto& \cases{f(e) \ \ \textnormal{if} \ e\in E \\ +\infty \ \ \textnormal{otherwise}}
\end{array}
$$

Now, with a slight abuse we have that $\big(Cl(\overline{G}, (-\overline{f})\big)$ is a filtered simplicial complex. For any weighted graph $(H, h)$ let the function $\beta^r_{(H, h)}$ be the persistent $r$-Betti number function of $\big(Cl(H), h_{Cl}\big)$.

\begin{definition}\label{extended}
The {\em extended persistent} $r${\em -Betti number function} of $(G, f)$ is $\overline{\beta}^r_{(G, f)}: \mathbb{R}^2\to \mathbb{Z}$, such that 
\[
\begin{array}{cccc}
&(u, v) & \mapsto &
\cases{\beta^r_{(G, f)}(u, v) \ \ \textnormal{if} \ \ u \le v\\
&\\
\beta^r_{\big(\overline{G}, -\overline{f}\big)}(-u, -v) \ \ \textnormal{if} \ \ u > v}
\end{array}
\]
\end{definition}

These functions are actually carrying more information than the ones relative just to $\big(Cl(G), f_{Cl}\big)$, as Fig.~\ref{extfig} shows on two filtrations of a same graph, which differ just in the last step.
The two filtrations are indistinguishable using just the persistent Betti functions on $\big(Cl(G), f_{Cl}\big)$: the persistent 0-Betti number functions are equal, and for $n\geq 1$ the persistent $n$-Betti number functions are trivially zero. On the contrary, as shown in the figure, the extended persistent 0-Betti number functions  are different.

\begin{remark}
Of course, the same extension can be defined also for the other constructions, i.e. by considering neighborhoods of $G$ and of $G^c$, enclaveless sets of $G$ and of $G^c$, etc. For the moment we restrict our attention to this case, because of the role of cliques and independent sets both in theory and applications.
\end{remark}

\section{Conclusions and future work}
The place occupied by graphs and networks in data representation and analysis is continuously gaining importance. Here, we showed how weighted graphs can be studied via the classical persistent homology paradigm, moving beyond the boundaries of the standard clique-based approaches.

We showed how persistent homology, combined with well known graph-theoretical concepts, can be used to reveal novel information about weighted graphs. Examples are given considering cliques, independent sets, neighborhoods and enclaveless sets. Furthermore, we explored the duality between cliques and independent sets by producing a Ramsey-inspired extended persistence.

The ability of measuring dissimilarity plays an essential role in all aspects of data analysis. Therefore, we have stressed the connection of these novel filtering functions and persistence diagrams with the natural pseudodistance.

\subsection{Possible developments}

After having defined the filtered complexes of Sect.~\ref{togrape} and their persistent homology, it would be important to study the information they convey in graph-theoretical terms.

We want to further investigate which complexes can be obtained from the considered constructions and which cannot.

Even in the cases when every sequence of Betti numbers can be produced by a certain construction (as is the case of Cor.~\ref{clbetti}), we need to know whether the natural pseudodistance between the resulting spaces coincides with the one between graphs, if we want to get an optimality (or universality) result like \cite[Thm. 5.5]{Les15}.

More structures in a graph respect the inheritance property, necessary for building a simplicial complex. We intend to proceed in examining them within the proposed framework.

A possible connection with \texttt{Mapper} \cite{SiMe*07} deserves attention.

We are working at a combinatorial, axiomatic definition of ``persistence functions'' which will enable direct use of persistence diagrams for weighted graphs, without passing through the construction of a simplicial complex.

\section*{Acknowledgments}
We are indebted to Diego Alberici, Emanuele Mingione, Pierluigi Contucci (whose research originated the present one), Luca Moci, Fabrizio Caselli, Patrizio Frosini and Pietro Vertechi for many fruitful discussions.
Article written within the activity of INdAM-GNSAGA.

\bibliographystyle{abbrv}
\bibliography{TopologicalGraphPersistence}

\begin{thebibliography}{10}

\bibitem{AlCo*17}
D.~Alberici, P.~Contucci, E.~Mingione, and M.~Molari.
\newblock Aggregation models on hypergraphs.
\newblock {\em Annals of Physics}, 376:412--424, 2017.

\bibitem{BoMu11}
A.~Bondy and U.~Murty.
\newblock {\em Graph Theory}.
\newblock Graduate Texts in Mathematics. Springer London, 2011.

\bibitem{ChaCo*09}
F.~Chazal, D.~Cohen-Steiner, M.~Glisse, L.~J. Guibas, and S.~Y. Oudot.
\newblock Proximity of persistence modules and their diagrams.
\newblock In {\em SCG '09: Proceedings of the 25th annual symposium on
  Computational geometry}, pages 237--246, New York, NY, USA, 2009. ACM.

\bibitem{CoEdHa07}
D.~Cohen-Steiner, H.~Edelsbrunner, and J.~Harer.
\newblock Stability of persistence diagrams.
\newblock {\em Discr.Comput. Geom.}, 37(1):103--120, 2007.

\bibitem{CoEdHa09}
D.~Cohen-Steiner, H.~Edelsbrunner, and J.~Harer.
\newblock Extending persistence using {P}oincar{\'e} and {L}efschetz duality.
\newblock {\em Foundations of Computational Mathematics}, 9(1):79--103, 2009.

\bibitem{dAFrLa06}
M.~d'Amico, P.~Frosini, and C.~Landi.
\newblock Using matching distance in size theory: A survey.
\newblock {\em Int. J. Imag. Syst. Tech.}, 16(5):154--161, 2006.

\bibitem{dAFrLa10}
M.~d'Amico, P.~Frosini, and C.~Landi.
\newblock Natural pseudo-distance and optimal matching between reduced size
  functions.
\newblock {\em Acta {A}pplicandae {M}athematicae}, 109(2):527--554, 2010.

\bibitem{EdHa08}
H.~Edelsbrunner and J.~Harer.
\newblock Persistent homology---a survey.
\newblock In {\em Surveys on discrete and computational geometry}, volume 453
  of {\em Contemp. Math.}, pages 257--282. Amer. Math. Soc., Providence, RI,
  2008.

\bibitem{EdHa09}
H.~Edelsbrunner and J.~Harer.
\newblock {\em Computational Topology: An Introduction}.
\newblock American Mathematical Society, 2009.

\bibitem{Ha02}
A.~Hatcher.
\newblock {\em Algebraic Topology}.
\newblock Algebraic Topology. Cambridge University Press, 2002.

\bibitem{HoMa*09}
D.~Horak, S.~Maleti{\'c}, and M.~Rajkovi{\'c}.
\newblock Persistent homology of complex networks.
\newblock {\em Journal of Statistical Mechanics: Theory and Experiment},
  2009(03):P03034, 2009.

\bibitem{HuRi17}
W.~Huang and A.~Ribeiro.
\newblock Persistent homology lower bounds on high-order network distances.
\newblock {\em IEEE Transactions on Signal Processing}, 65(2):319--334, 2017.

\bibitem{Jo08}
J.~Jonsson.
\newblock {\em Simplicial complexes of graphs}, volume~3.
\newblock Springer, 2008.

\bibitem{Les15}
M.~Lesnick.
\newblock The theory of the interleaving distance on multidimensional
  persistence modules.
\newblock {\em Foundations of Computational Mathematics}, pages 1--38, 2015.

\bibitem{Lo78}
L.~Lov{\'a}sz.
\newblock Kneser's conjecture, chromatic number, and homotopy.
\newblock {\em Journal of Combinatorial Theory, Series A}, 25(3):319--324,
  1978.

\bibitem{MaRa*08}
S.~Maleti{\'c}, M.~Rajkovi{\'c}, and D.~Vasiljevi{\'c}.
\newblock Simplicial complexes of networks and their statistical properties.
\newblock In {\em International Conference on Computational Science}, pages
  568--575. Springer, 2008.

\bibitem{MaZh*16}
S.~Maleti{\'c}, Y.~Zhao, and M.~Rajkovi{\'c}.
\newblock Persistent topological features of dynamical systems.
\newblock {\em Chaos: An Interdisciplinary Journal of Nonlinear Science},
  26(5):053105, 2016.

\bibitem{PaMo*17}
S.~Pal, T.~J. Moore, R.~Ramanathan, and A.~Swami.
\newblock Comparative topological signatures of growing collaboration networks.
\newblock In {\em Workshop on Complex Networks CompleNet}, pages 201--209.
  Springer, 2017.

\bibitem{PeEx*14}
G.~Petri, P.~Expert, F.~Turkheimer, R.~Carhart-Harris, D.~Nutt, P.~J. Hellyer,
  and F.~Vaccarino.
\newblock Homological scaffolds of brain functional networks.
\newblock {\em Journal of The Royal Society Interface}, 11(101):20140873, 2014.

\bibitem{Ra30}
F.~P. Ramsey.
\newblock On a problem of formal logic.
\newblock {\em Proceedings of the London Mathematical Society},
  s2-30(1):264--286, 1930.

\bibitem{ReNo*17}
M.~W. Reimann, M.~Nolte, M.~Scolamiero, K.~Turner, R.~Perin, G.~Chindemi,
  P.~D{\l}otko, R.~Levi, K.~Hess, and H.~Markram.
\newblock Cliques of neurons bound into cavities provide a missing link between
  structure and function.
\newblock {\em Frontiers in Computational Neuroscience}, 11:48, 2017.

\bibitem{SiMe*07}
G.~Singh, F.~M{\'e}moli, and G.~E. Carlsson.
\newblock Topological methods for the analysis of high dimensional data sets
  and 3d object recognition.
\newblock In {\em SPBG}, pages 91--100, 2007.

\bibitem{SiGi*18}
A.~E. Sizemore, C.~Giusti, A.~Kahn, J.~M. Vettel, R.~F. Betzel, and D.~S.
  Bassett.
\newblock Cliques and cavities in the human connectome.
\newblock {\em Journal of Computational Neuroscience}, 44(1):115--145, Feb
  2018.

\bibitem{Sp94}
E.~H. Spanier.
\newblock {\em Algebraic topology}, volume~55.
\newblock Springer Science \& Business Media, 1994.

\bibitem{VeUrFrFe93}
A.~Verri, C.~Uras, P.~Frosini, and M.~Ferri.
\newblock On the use of size functions for shape analysis.
\newblock {\em Biol. Cybern.}, 70:99--107, 1993.

\end{thebibliography}

\end{document}